\documentclass[10pt]{elsarticle}
\usepackage[utf8]{inputenc}
\usepackage[T1]{fontenc}
\usepackage{amsmath}
\usepackage{amsfonts}
\usepackage{graphicx} 
\usepackage{xcolor}

\DeclareMathOperator*{\argmin}{arg\,min}

\usepackage{cleveref}
\newtheorem{theorem}{Theorem}

\newtheorem{remark}{Remark}
\newtheorem{example}{Example}
\newenvironment{proof}{{\bf Proof.}}{\hfill$\square$}

\newcommand{\Vv}{L^2(\Gamma) \cap \beta H_0^2(\Omega)}
\newcommand{\VV}{V_\beta}
\newcommand{\VVh}{V_{\beta,h}}
\newcommand{\GG}{G_\beta}
\newcommand{\GGh}{G_{\beta,h}}
\newcommand{\Gg}{\beta L^2(\Omega)}
\newcommand{\MM}{M_{\beta}}
\newcommand{\MMh}{M_{\beta,h}}
\newcommand{\Mm}{\frac{1}{\beta} L^2(\Omega)}
\newcommand{\partialn}{\partial_{\mathbf{n}}}

\newcommand{\Va}{V_{\alpha}}
\newcommand{\VA}{\alpha L^2(\partial \Omega) \cap H^2(\Omega)}

\title{Variational Formulations of the Strong Formulation - Forward and Inverse Modeling
using Isogeometric Analysis and Physics-Informed Networks}

\author[a,b]{Kent-Andre Mardal}
\author[a]{Jarle Sogn}
\author[b]{Marius Zeinhofer}

\affiliation[a]{{Simula Research Laboratory, Kristian Augusts gate 23, 0164 Oslo, Norway}}
\affiliation[b]{{Department of Mathematics, University of Oslo, Oslo, Norway}}

\begin{document}

\maketitle

\section{Introduction}

The recently introduced  Physics-Informed Neural Networks (PINNs)~\cite{dissanayake1994neural, raissi2019physics, karniadakis2021physics}
have popularized least squares formulations of both forward and inverse problems
involving partial differential equations (PDEs) \emph{in strong form}. In this brief contribution we discuss the well-posedness of  both cases with PDEs of elliptic type. We address  weighting parameters, solution regularity and partial input data. Specifically, we discuss the following forward problem: Find $u$ such that 
\begin{align}
\label{L1}
u &= \argmin_v \ L_1(v; f, g) \quad \mbox{with }\\ 
L_1(v; f, g) &= \|-\Delta v - f\|_{L^2(\Omega)}^2 +  \alpha^2 \|v - g\|_{L^2(\partial \Omega)}^2,   
\end{align}
where $u$ is the unknown minimizer, $f, g$ are the body and surface forces, respectively and $\Omega$ is a bounded domain in $R^n$. 
Here, it has been observed that for some problems $\alpha$ needs
to be chosen carefully~\cite{mcclenny2020self, wang2021understanding, wang2022and}.  

The second problem we will consider is the following
inverse problem:  
\begin{align}
\label{L2}
u, f &= \argmin_{v,g} \ L_2(v, g; u_d, f_p) \mbox { with }\\ 
\ L_2(v, g; u_d, f_p) &= \gamma^2 \|-\Delta v - g\|_{L^2(\Omega)}^2 +  \|v - u_d\|_{L^2(\Gamma)}^2 + \beta^2 \|g-f_p\|_{L^2(\Omega)}^2 .    
\end{align}
Here $u$ and  $f$ are the unknown state and control, $u_d$ embodies state observations while $f_p$ encodes prior knowledge of the control variable. We will in particular consider \emph{partial observations}, meaning that $\Gamma$ is a proper subdomain of $\Omega$, and discuss the role of $\beta$ and $\gamma$. For this problem, we will assume that the Dirichlet boundary conditions for $v$ are exactly satisfied by the discrete ansatz.


The weighted Sobolev spaces are defined as follows. Let $X$ be a bounded domain
in $\mathbb{R}^d$. Then $L^2(X)$ is the space of square integrable functions on $X$ and $H^m(X)$ consists of functions with $m$ derivatives in $L^2(X)$. Finally, the intersection space
$\mu H^m(X) \cap \lambda H^n(Y)$ for $\mu, \lambda \in \mathbb{R}$ is defined as
\[
\mu H^m(X) \cap \lambda H^n(Y) = \{ u \ | \ \mu^2 \|u\|^2_{H^m(X)} +  \lambda^2 \|u\|^2_{H^n(Y)} < \infty \}
\]
We let $(\cdot, \cdot)_X$ denote the $L^2$ inner product on $L^2(X)$ and set $L^2(X) = H^0(X)$. For the intersection spaces below we will, with a slight abuse of notation, use the semi-norm of the Sobolev space with the highest regularity. For instance, the inner product of 
$\alpha L^2(\partial \Omega) \cap H^2(\Omega)$ is here defined as
\[
(u, v)_{ \alpha L^2(\partial \Omega) \cap H^2(\Omega)} = \alpha^2 (u, v)_{L^2(\partial \Omega)} + (\Delta u, \Delta v)_{L^2(\Omega)} .
\]
We remark that $\alpha L^2(\partial \Omega) \cap H^2(\Omega)$ is not norm equivalent to the standard $H^2$ space. 
However, in \cite{zeinhofer2023unified} it is shown that it holds $\|u\|_{\alpha L^2(\partial \Omega) \cap H^2(\Omega)} \ge C\|u\|_{H^{1/2}(\Omega)}$ which implies that $\alpha L^2(\partial \Omega) \cap H^2(\Omega)$ is a Hilbert space. We also remark that approximation properties are typically stated in semi-norms rather than full norms~\cite{bramble1970estimation} and as such the error rates are related to the highest order derivatives.


Below we will also use $\partialn$ to denote $ \mathbf{n} \cdot \nabla$. 
Finally, $H^2_0$ denotes the subspace of $H^2$ with trace zero. 
We will in the following assume the following regularity for the problem $-\Delta u  = f$ in $H^2_0$ ~\cite{grisvard2011elliptic}: 
\begin{equation}
\label{H2regularity}
\|u\|_{H^2(\Omega)} \le C \|f\|_{L^2(\Omega)}.
\end{equation}

\section{Forward problem}
Let $\Va=\VA$. 
A straightforward calculation shows that the minimizer of $L_1$ 
is the solution of the following variational problem: 
Find $u \in \Va$ such that
\begin{equation}
\label{L1var}
(\Delta u, \Delta v)_\Omega + \alpha^2 (u,v)_{\partial \Omega} = 
(f, -\Delta v)_{\Omega} + \alpha^2 (g, v)_{\partial \Omega} \quad \forall v \in \Va . 
\end{equation}
The corresponding strong problem is a bi-harmonic problem that reads: 
Find $u$ such that 
\begin{align}
\label{L1strong1Alt}
\Delta^2 u &= -\Delta f, \quad &\mbox{ in } \ \Omega, \\
\label{L1strong2Alt}
         \alpha^2 u - \partialn \Delta u &= \alpha^2 g + \partialn f, \quad &\mbox{ on } \ \partial \Omega, \\
\label{L1strong3Alt}
         -\Delta u &= f, \quad &\mbox{ on } \ \partial \Omega. 
\end{align}

\begin{theorem}
The problem \cref{L1} or equivalently \cref{L1var}, corresponds to the bi-harmonic problem 
\cref{L1strong1Alt}-\cref{L1strong3Alt} and is well-posed
in $\Va$. 
\end{theorem}
\begin{proof}
The result is a direct consequence of Green´s lemma.  
\begin{align*}
(-\Delta u -f, -\Delta v)_\Omega +
 \alpha^2 (u -g, v)_{ \partial \Omega}   
 &= 
 - (-\Delta^2 u - \Delta f, v)_\Omega 
     -  (-\Delta u -f,  \partialn v)_{\partial \Omega} \\
    &- (\partialn (-\Delta u -f ) 
    +  \alpha^2 (u -g ), v )_{\partial \Omega} . 
\end{align*} 
The coercivity of \cref{L1var} follows as the bilinear form is an inner product and the right-hand side clearly lies in the dual space of $\Va$.
\end{proof}

Hence, we remark that in \cref{L1strong2Alt} a weighted average of the 
boundary condition and normal derivative of the PDE at the boundary  is enforced. 

\section{Inverse Problem}
We discussed the handling of boundary conditions in the previous 
section so let us here simplify the situation and only 
consider homogeneous Dirichlet conditions, which we assume to be satisfied exactly by the ansatz class. 
We recall that $\Gamma$ is a subdomain of $\Omega$. 
Further, let
 $\VV=\Vv$,  $\GG=\Gg$, and $\MM=\Mm$. 
Then, the minimizers $u,f$ of $L_2$ in \cref{L2} satisfy  the following variational problem: 
Find $u\in\VV,f\in\GG \in $ such that 
\begin{align}
\label{L2var1}
(u, v)_\Gamma + \gamma^2 (-\Delta u -f , -\Delta v)_\Omega &= ( u_d, v)_{\Gamma},  &\forall v\in\VV    \\
\label{L2var2}
\gamma^2(-\Delta u  , -g)_\Omega + (\beta^2 + \gamma^2) (f, g)_\Omega         &= \beta^2 (f_p, g)_\Omega, &\forall g\in\GG
\end{align}

The minimizers $u\in \VV,f\in \GG$ of \cref{L2var1}-\cref{L2var2} satisfy  the following penalized Lagrangian problem problem: 
Find $u,f,\lambda \in $ such that
\begin{align}
\label{L2var1Lag}
(u, v)_\Gamma +  (\lambda, -\Delta v)_\Omega &= (u_d, v)_\Gamma, \quad &\forall v\in \VV \\
\label{L2var2Lag}
\beta^2 (f, g)_\Omega  + (\lambda, -g)_\Omega     &= \beta^2 (f_p, g)_\Omega, \quad &\forall g\in \GG  \\
\label{L2var3Lag}
(-\Delta u -f , \mu)_\Omega - \frac{1}{\gamma^2} (\lambda, \mu)_\Omega       &= 0, \quad &\forall \mu \in \MM 
\end{align}
It is straightforward to obtained \cref{L2var1}-\cref{L2var2} by eliminating $\lambda$
from \cref{L2var1Lag}-\cref{L2var3Lag}.

\begin{theorem}
\label{wellposedinverse}
The problem \cref{L2}, \cref{L2var1}-\cref{L2var2}, or \cref{L2var1Lag}-\cref{L2var3Lag}   is well-posed if $1/\gamma < \beta$
and if the solution $u$ admits $H^2$ regularity. 
\end{theorem}

\begin{proof}
We first consider the case $\gamma=\infty$. The saddle-point problem must then satisfy the
four Brezzi conditions~\cite{brezzi1974existence}. The two conditions of continuity follow from a straightforward calculation. Hence, we will only 
consider the coercivity and inf-sup conditions. 
The coercivity follows from $H^2$ regularity \cref{H2regularity}: 
\[
(u, u)_{L^2(\Gamma)} + \beta^2(f,f)_{L^2(\Omega)} \ge (u, u)_{L^2(\Gamma)} + \frac12 \beta^2(f,f)_{L^2(\Omega)} + \frac{1}{2C^2} \beta^2(u,u)_{H^2(\Omega)}
\]
The inf-sup condition can be shown with $\hat u = 0$ and $\hat f = \mu$
\begin{align*}
\sup_{u,f} \frac{(-\Delta u - f, \mu)}{(\|u\|^2_{\VV} + \|f\|^2_{\GG})^{1/2}} \ge \\
\frac{(-\Delta \hat u - \hat f, \mu)}{(\|\hat u\|^2_{\VV} + \|\hat f\|^2_{\GG})^{1/2}} =
\frac{(\mu, \mu)}{(\|\mu \|^2_{\GG})^{1/2}} = \|\mu\|_{\MM}
\end{align*}
For $\gamma < \infty$ the theory of Braess \cite{braess1996stability} states that a saddle point problem with a penalty term is well posed as long as the penalty is bounded by the norm of the Lagrange multiplier, which is the case when  $1/\gamma < \beta$. 
\end{proof}

\begin{remark}
The above proof follows~\cite{mardal2017robust, mardal2022robust} and employs \cite{braess1996stability} for
the stability estimates involving the penalty term. However, using the more general theory~\cite{hong2021new, sogn2019schur}, the requirement
that $1/\gamma < \beta$ is removed if the norms $\VV = L^2(\Gamma)\cap (\frac{1}{\beta^2} + \frac{1}{\gamma^2})^{-1/2}H^2(\Omega)$ and 
 $\MM = (\frac{1}{\beta^2} + \frac{1}{\gamma^2})^{1/2}L^2(\Omega)$  
are used.  
\end{remark}

Let us consider conforming discretizations of the problem \cref{L2var1Lag}-\cref{L2var3Lag}
such that 
$\VVh\subset \VV$, 
$\GGh\subset \GG$, and 
$\MMh\subset \MM$. Furthermore, let 
\[
Z = \left\{ u\in\VV, f\in\GG \ | \ (-\Delta u - f, \mu) = 0, \forall \mu\in\MM \right\}
\]
and 
\[
Z_h = \left\{ u\in \VVh, f\in \GGh \ | \ (-\Delta u - f, \mu) = 0, \forall \mu\in \MMh \right\}
\]
For conforming discretizations in which $Z_h \subset Z$ or in our case $\Delta \VVh \subset \MMh$
the \Cref{wellposedinverse} applies, i.e., the same proof can be carried out in in the discrete case.


Finally, we remark that there are some subtleties of our approach of proving the well-posedness 
of discretizations of \cref{L2} via the Lagrange formulation~\cref{L2var1Lag}-\cref{L2var3Lag}.
A discrete version of \cref{L2var1Lag}-\cref{L2var3Lag} reads: 
\[
\begin{pmatrix}
M_{\Gamma} & & A^T \\
& \beta^2 M & M \\
A & M & -\frac{1}{\gamma^2}M
\end{pmatrix}
\begin{pmatrix}
u \\ f \\ \lambda 
\end{pmatrix}
=
\begin{pmatrix}
u_d \\ f_p \\ 0
\end{pmatrix} . 
\]
Here, $A$ is the stiffness matrix, $M$ is the mass matrix and $M_\Gamma$ is the mass matrix 
related to the partial observation. Note that $A$ is not symmetric as the trial and test functions differs. However, $M$ is symmetric. 
Eliminating $\lambda$ yields
\begin{equation}
\label{L2disc}
\begin{pmatrix}
M_{\Gamma} + \gamma^2 A^T M^{-1} A & \gamma^2 A^T  \\
\gamma^2 A  & (\beta^2 + \gamma^2) M \\
\end{pmatrix}
\begin{pmatrix}
u \\ f 
\end{pmatrix}
=
\begin{pmatrix}
u_d \\ f_p 
\end{pmatrix} . 
\end{equation}

Clearly, a discretization of \cref{L2var1}-\cref{L2var2} would appear similar to \cref{L2disc}, 
except for the inverted mass matrix. 
In most cases, e.g., continuous finite elements, isogeometric analysis or neural networks, the mass matrices are far from diagonal and simply replacing the inverted mass matrix with, for instance, a lumped variant results in loss of approximation. An alternative and more practical condition is that $-\Delta \VVh \subset \GGh$. Then $A M^{-1} A = \Delta^2$  on $\VVh$. 
This can be shown as follows: The upper bound 
\[
(A M^{-1} A u, u) = \sup_{\lambda \in G_h}\frac{(Au, \lambda)^2}{(M\lambda,\lambda)} = \sup_{\lambda \in G_h}\frac{(-\Delta u, \lambda)^2}{\|\lambda\|^2_{L^2(\Omega)}} \leq \|\Delta u\|^2_{L^2(\Omega)}.
\]
Furthermore, since $-\Delta u \in G_h$, we can choose $\lambda = -\Delta u$ and obtain a lower bound
\[
(A M^{-1} A u, u) = \sup_{\lambda \in G_h}\frac{(-\Delta u, \lambda)^2}{\|\lambda\|^2_{L^2(\Omega)}} \geq \frac{(-\Delta u, -\Delta u)^2}{\|\Delta u\|^2_{L^2(\Omega)}} ) = \|\Delta u\|^2_{L^2(\Omega)}.
\]

\section{Numerical Examples}

In this section we illustrate the theory with some numerical examples. 
In Example 1 we consider a discretization of the forward problem using 
isogeometric analysis (IGA)~\cite{cottrell2009isogeometric}. IGA easily allow for $H^2$ conforming discretizations. In  \Cref{t:table3}
we see that approximation rates corresponding to $H^2$ are obtained.  
Then in the Examples 2 and 3 we solve the inverse problem with IGA and neural networks (NNs), respectively. We remark that $-\Delta \VVh \subset \GGh$
is feasible both using IGA and NNs. In \Cref{t:table1dControl} and \Cref{t:table2dControl} we compare IGA with and without the $-\Delta \VVh \subset \GGh$ and note that approximations that fulfill this condition are superior. The same conclusion can be drawn in Example 3 where NNs with and without this condition are compared.  
We remark that in Example 2 and 3, for the purpose of investigating
the numerical properties of the system, we let the prior $f_p$ be the source function that corresponds with the exact solution. Such priors are rarely known in practice and our example should as such be classified as an inverse crime. 
However, as our interest here is to investigate the potential of the numerical scheme, rather than a concrete application, 
we allow ourselves to perform such a crime as the quality of the prior may vary. 
We remark that we have used long double for the IGA calculations as the condition number of the system is large. 

\begin{example}
\label{ex1}
Let $\Omega = (0,1)^2$ and consider the variational problem \cref{L1var}
by using tensor product B-splines  $S_{p,\ell, q}((0,1))$, where $p$ is the spline degree, $\ell$ is the number of uniform refinements, and $q=p-1$ is the continuity. The grid-size is $h = 2^{-\ell}$. For a problem with the following manufactured solution
\[
g = u = \cos(k\pi x) \cos(k\pi y), \quad f =\Delta u = -2k^2\pi^2\cos(k\pi x) \cos(k\pi y).
\]

In  \Cref{t:table3} we see that the error is virtually independent of $\alpha$
and that we get the expected linear convergence in $H^2$ for $p=2$. For $p=5$ 
we obtained forth order convergence (not shown) only for $\alpha>10^3$, but
all $\alpha$ has higher than third order convergence. 
\end{example}

\begin{table}[ht]
\caption{Error in full $H^2$ norm in \Cref{ex1},  $k=1$, $p=2$ except (last row $p=5$)}
\label{t:table3}
\begin{center}
  \begin{tabular}{| c || c | c | c | c | c || l |}
    \hline
    $\ell \;\backslash\; \alpha^2$ & $10^6$ & $10^{3}$ & $10^{0}$ & $10^{-3}$ & $10^{-6}$ & DoF \\ \hline \hline
    $3$ &  7.84e-02 & 7.85e-02 & 7.86e-02 & 7.86e-02 & 7.86e-02 & 100\\  \hline
    $4$ &  3.91e-02 & 3.91e-02 & 3.91e-02 & 3.91e-02 & 3.91e-02 & 324\\  \hline
    $5$ &  1.95e-02 & 1.95e-02 & 1.95e-02 & 1.95e-02 & 1.95e-02 & 1156\\  \hline
    $6$ &  9.76e-03 & 9.76e-03 & 9.76e-03 & 9.76e-03 & 9.76e-03 & 4356\\  \hline
    $6 \; (5)$ & 3.60e-09 & 8.22e-09 & 1.06e-07 & 1.10e-06 & 1.39e-05 & 4761\\  \hline
\end{tabular}$\quad$
\end{center}\end{table}

\begin{example}
\label{prob:2DInverse3by3Lim}
Let $\Gamma = (0.25,0.75)^2$ and $\Omega = (0,1)$, let $g = u_d$ and $u_d = \sin(\pi k x)\sin(\pi k y)$. Also let $f_p = 2\pi^2 k^2\sin(\pi k x)\sin(\pi k y)$. Let $U_h := \{u\in S_{p,\ell,p-1}(\Omega)\,:\, u|_{\partial \Omega} = 0\}$ and let $F_h$ be $F_h = S_{p,\ell,p-1}(\Omega)$ or $F_h = S_{p,\ell,p-3}(\Omega)$ ($\Delta U_h \subset F_h $)
and solve \eqref{L2var1}.

\Cref{t:table1dControl} and \Cref{t:table2dControl} shows the error of the state
in $H^2$ norm with and without the $\Delta U_h \subset F_H $ condition. 
Clearly, discretizations that satisify this condition has improved approximation. 
\end{example}

\begin{table}[ht]
\caption{$|u_h - u_d|_{H^2(\Omega)}$ in \Cref{prob:2DInverse3by3Lim}, $\ell =6$, $k=1$, $p=2$, $\Delta U_h\not\subset F_h$.  }
\label{t:table1dControl}
\begin{center}
  \begin{tabular}{| c || c | c | c |}
    \hline
    $\beta^2 \;\backslash\; \gamma^2$ & $10^0$ & $10^{2}$ & $10^{4}$  \\ \hline \hline
    $1$       & 9.78e-03 & 3.10e-02 & 0.291 \\  \hline
    $10^{-2}$ & 3.06e-02 & 0.290 & 0.335 \\  \hline
    $10^{-4}$ & 0.304    & 0.363 &0.364 \\  \hline
\end{tabular}$\quad$
\end{center}\end{table}

\begin{table}[ht]
\caption{$|u_h - u_d|_{H^2(\Omega)}$ in \Cref{prob:2DInverse3by3Lim}, $\ell =6$, $k=1$, $p=2$, $\Delta U_h\subset F_h$.}
\label{t:table2dControl}
\begin{center}
  \begin{tabular}{| c || c | c | c |}
    \hline
    $\beta^2 \;\backslash\; \gamma^2$ & $10^0$ & $10^{2}$ & $10^{4}$  \\ \hline \hline
    $1$       & 9.76e-03 & 9.76e-03 & 9.76e-03 \\  \hline
    $10^{-2}$ & 9.76e-03 & 9.76e-03 & 9.76e-03 \\  \hline
    $10^{-4}$ & 9.76e-03 & 9.76e-03 & 9.88e-03 \\  \hline
\end{tabular}$\quad$
\end{center}\end{table}

\begin{example}
\label{ex3}
In our third example we consider neural networks. Concretely, we consider shallow neural networks of width $n\in\mathbb N$, with $\tanh$ activation function, i.e., 
\[
u_\theta(x) = \sum_{i=1}^n a_i\tanh(w_i\cdot x + b_i) + c
\]
where $\theta$ consists of the trainable weights $a_i, w_i, b_i$ and $c$ for $i=1,\dots, n$. In the case of the inverse problem both $u_\theta$ and $f_\psi$ are parametrized as neural networks. To guarantee that $\Delta u_\theta$ can be expressed as a neural network, we use $\tanh''$ as activation function for the architecture employed for $f_\psi$. For the optimization, we rely on the recently proposed energy natural gradient method \cite{pmlr-v202-muller23b}. We handle the boundary conditions using the penalty approach of \cref{L1} with $\alpha=1$.
\end{example}

\begin{table}[ht]
\caption{Neural network approach using two shallow networks of width 32 for both $u$ and $f$ in \Cref{ex3}. In both cases $\operatorname{tanh}$ is used as an activation function, hence the Laplacian of $u_\theta$ is not representative by the network $f_\psi$. We report relative $H^2$ errors.}
\label{t:networks_no_containment}
\begin{center}
  \begin{tabular}{| c || c | c | c |}
    \hline
    $\alpha^2 \;\backslash\; \gamma^2$ & $10^0$ & $10^{2}$ & $10^{4}$  \\ \hline \hline
    $1$       & 1.32e-05 & 1.08e-04 & 9.27e-04 \\  \hline
    $10^{-2}$ & 2.79e-05 & 4.11e-05 & 5.95e-04 \\  \hline
    $10^{-4}$ & 5.63e-05 & 3.07e-04 & 4.38e-03 \\  \hline
\end{tabular}$\quad$
\end{center}\end{table}

\begin{table}[ht]
\caption{\Cref{ex3}: Neural network approach using two shallow networks of width 32 for both $u$ and $f$. We use $\operatorname{tanh}$ as activation function for $u_\theta$ and $\operatorname{tanh}'' = 2\operatorname{tanh}^3 - 2\operatorname{tanh}$ for $f_\psi$, hence the Laplacian of $u_\theta$ is representable by the network $f_\psi$. We report relative $H^2$ errors.}
\label{t:networks_containment}
\begin{center}
  \begin{tabular}{| c || c | c | c |}
    \hline
    $\alpha^2 \;\backslash\; \gamma^2$ & $10^0$ & $10^{2}$ & $10^{4}$  \\ \hline \hline
    $1$       & 1.34e-05 & 7.09e-05  & 3.05e-04 \\  \hline
    $10^{-2}$ & 1.44e-05 & 5.34e-05  & 3.56e-04 \\  \hline
    $10^{-4}$ & 3.17e-05 & 1.78e-04  & 7.30e-04 \\  \hline
\end{tabular}$\quad$
\end{center}\end{table}

\section{Acknowledgement}
The authors want to thank Miroslav Kuchta and Bastian Zapf
for interesting discussions on the topic
as well as support from the Research Council of Norway 
through the grants 300305 and 301013.

\bibliography{references}
\bibliographystyle{abbrvnat}

\end{document}